\documentclass[11pt,twoside,reqno]{amsart}
\usepackage{amssymb}
\usepackage{enumerate}
\usepackage{amsmath}
\usepackage{a4wide}
\usepackage[usenames]{color}
\usepackage[dvipsnames]{xcolor}
\usepackage{soul}

\usepackage[hidelinks=true]{hyperref} 

\usepackage{mathtools}

\usepackage{lmodern}
\usepackage[utf8]{inputenc}
\usepackage[L7x]{fontenc} 
\usepackage{blkarray}

\newcommand{\abs}[1]{\lvert#1\rvert}

\usepackage[
   backend=bibtex, 
   style=numeric,
   isbn=false,
   giveninits=true,
   doi=false,
   url=false]{biblatex}
\newtheorem{theorem}{Theorem}

\newtheorem{lemma}[theorem]{Lemma}

\theoremstyle{remark}

\newtheorem*{acknowledgment}{\textbf{Acknowledgments}}

\def\leq{\leqslant}
\def\geq{\geqslant}
\def\al{\alpha}
\def\be{\beta}
\def\ga{\gamma}

\def\F{\mathbb F}
\def\N{\mathbb N}

\def\Q{\mathbb Q}
\def\Z{\mathbb Z}

\DeclareMathOperator{\ord}{ord}

\begin{document}

\title{No three algebraic conjugates of degree sixteen sum to zero}

\author{Žygimantas Baronėnas, Paulius Drungilas, and Jonas Jankauskas}

\address{Institute of Mathematics, Faculty of Mathematics and Informatics, Vilnius
University, Naugarduko 24, Vilnius LT-03225, Lithuania}
\email{zygimantas.baronenas@mif.stud.vu.lt}

\address{Institute of Mathematics, Faculty of Mathematics and Informatics, Vilnius
University, Naugarduko 24, Vilnius LT-03225, Lithuania}
\email{paulius.drungilas@mif.vu.lt}

\address{Institute of Mathematics, Faculty of Mathematics and Informatics, Vilnius
University, Naugarduko 24, Vilnius LT-03225, Lithuania}
\email{jonas.jankauskas@mif.vu.lt}


\begin{abstract}
Let $d$ be the smallest positive integer, not a multiple of $3$, for which there exists an algebraic number $\al$ of degree $d$ over $\mathbb{Q}$ whose three algebraic conjugates add to zero. We prove that $d=20$. 
This is derived from the following result: 
for any linear relation $\sum_{j=1}^d a_j \al_j=0$ with coefficients $a_j\in\mathbb{Z}$ among the conjugates $\al_j$ of an algebraic number of degree $d=p^m$, where $p$ is a prime number, $m \geq 1$, the sum $\sum_{j=1}a_j$ is divisible by $p$. 
If $d=2p^m$, $p\geq 3$ and $\sum_{j=1}^d|a_d| < p$, then $\sum_{j=1}a_j$ is an even number.
\end{abstract}
\maketitle

\section{Introduction}\label{intro}

In 2004 Dubickas and Smyth \cite{DubickasSmyth2006} asked to prove or disprove the following: If $\alpha_1+\alpha_2+\alpha_3=0$ for three distinct algebraic conjugates $\al_1$, $\al_2$, $\al_3$ of an algebraic number $\alpha$ of degree $d$, then 3 divides $d$. 
Stong (see \cite{DubickasSmyth2006}) found a counterexample of degree $d=20$. He proved that the irreducible polynomial
\begin{equation}\label{example20}
f(t) = t^{20} + 4\cdot5^{9}\cdot t^{10} + 16\cdot5^{15}
\end{equation}
has three distinct roots that add to zero. 
Several authors (see, e.g., \cite{DubickasJankauskas2015,Virbalas2025a}) were interested in the following natural question: what is the smallest positive integer $d$, not a multiple of 3, for which there exists an algebraic number of degree $d$ such that some three of its conjugates sum to zero? The above-mentioned counterexample of Stong implies that $d\leq 20$. Dubickas and Jankauskas (see Theorem 1.2 in \cite{DubickasJankauskas2015}) showed that such a minimal value of $d$ lies in the range $10\leq d \leq 20$. By the result of \cite{Kurbatov1977}, $d$ cannot be a prime number. Thus, $d=10, 14, 16$ or $20$. Recently, Virbalas (see Theorem 1.1 in \cite{Virbalas2025a}) showed that $d\neq 2p$, where $p\geq 5$ is a prime number. Thus, either $d=16$ or $d=20$. The first result of the present paper states that $d=20$:

\begin{theorem}\label{p1}
Let $d>1$ be the smallest positive integer, not a multiple of $3$, for which there exists an algebraic number of degree $d$ whose three algebraic conjugates sum to zero. Then $d=20$.
\end{theorem}

The case $d=16$ is ruled out by setting in the parameter values $p=2$ and $m=4$ in Theorem~\ref{p2} which directly follows from Theorem~\ref{t1} and Theorem~\ref{t2}.  Theorem~\ref{p2} generalizes the recent result of Virbalas \cite{Virbalas2025a}:

\begin{theorem}\label{p2}
Let $f(t)$ be an irreducible polynomial of degree $d$ over the field $\Q$. If the degree $d$ is of the form $d=p^m$, where the prime number $p \ne 3$, or $d$ is of the form $d=2p^m$ for a prime $p\geq5$ with integer exponents $m \geq 1$, then for any three roots $\al_1, \al_2, \al_3$ of $f(t)$, $\al_1 + \al_2 \pm \al_3 \ne 0$.
\end{theorem}

The relations of the form $\alpha_1+\alpha_2 \pm \alpha_3=0$ between three algebraic conjugates of $\al$ is a particular case of a more general class of relations
\begin{equation}\label{eqin1}
a_1\alpha_1+a_2\alpha_2+\dotsb+a_d\alpha_d=0,
\end{equation}
where $\alpha_1,\alpha_2,\dotsc,\alpha_d$ are the algebraic conjugates (over $\Q$) of an algebraic number $\alpha$ of degree $d$ and $a_1,a_2,\dotsc,a_d$ are rational integers, not all zero. If $a_1=a_2=\dotsc=a_d$, the relation \eqref{eqin1} is called \textit{trivial}. One of the first general results was the aforementioned result of Kurbatov \cite{Kurbatov1977}, who proved that there are no non-trivial additive relations of the form \eqref{eqin1} if the degree $d=p$, $p$ -- a prime (see also \cite{Dixon1997} and \cite{Baron1995}). Girstmair in \cite{Girstmair1999} proposed a theoretical framework to study linear relations \eqref{eqin1} via the representations of the Galois group of the polynomial $f(t)$. A lot of attention was devoted to the investigation of the relations \eqref{eqin1} when both the degree $d$ and the coefficients of the linear relation are small, say, $|a_1|+|a_2|+\dotsc+|a_d| \leq 4$, $d \leq 8$. Refer to \cite{DubickasHareJankauskas2017,DubickasJankauskas2015, Girstmair1982, Girstmair2006, Girstmair2007, Girstmair2008, Lalande2010, Lalande2007} for the results related to the linear relation $\alpha_1+\alpha_2 \pm \alpha_3 = 0$,  see \cite{Baronenas2026} for the results related to the linear relations $\alpha_1=\alpha_2+\alpha_3+\alpha_4$ and $\alpha_1+\alpha_2=\alpha_3+\alpha_4$. For the classification of all possible relations in case when $d=4$ refer to \cite{DuVi25,Kitaoka2017}. Multiplicative analogs of the relations \eqref{eqin1} were investigated in \cite{DuVi25} and \cite{Serrano2025}.

The main contribution of a present paper to the literature on the linear relations of type \eqref{eqin1} are the next two results on the divisibility property of the coefficient sums, from which both Theorems \ref{p1} and \ref{p2} are derived. Theorem \ref{t1} concerns the case where the degree $d$ is the power of a prime.

\begin{theorem}\label{t1}
Let $\al_1$, $\al_2$, $\dots$, $\al_d$ be the roots of the irreducible polynomial $f(t) \in \Q[t]$ of degree $d = p^m$, where $p \in \Z$ is a prime number, $m=\ord_p(d) \geq 1$. Suppose that there are integers $a_j \in \mathbb{Z}$, $1 \leq j \leq d$, such that $\sum_{j=1}^d a_j \al_j = 0$.
Then $p$ divides $\sum_{j=1}^d a_j$.
\end{theorem}

Theorem \ref{t2} deals with the case where the degree $d$ is twice the power of an odd prime.

\begin{theorem}\label{t2}
Let $f(t)$ be an irreducible polynomial over $\mathbb{Q}$ of degree $d = 2p^m$ with $p \ge 3$ being a prime. Then, for any linear relation  $\sum_{j=1}^d a_j \al_j = 0$ among the roots $\al_j$ of $f(t)$ with coefficients $a_j \in \Z$ that satisfy $\sum_{j=1}^d|a_d| < p$, the sum $\sum_{j=1}^d a_j$ is even.
\end{theorem}

This paper is organized as follows. In Section~\ref{auxiliary} we state two lemmas which are used in the proofs of Theorem~\ref{t1} and \ref{t2}. The preliminary setup before the proofs of these two theorems is given in Section~\ref{setup}. Section~\ref{proofs} contains the proofs of Theorem~\ref{p1}, \ref{t1} and \ref{t2}. 
In Section~\ref{open} we discuss possible values of $d$ up to 100 for which there exists an algebraic number of degree $d\geq 3$ whose three algebraic conjugates sum to zero.

\section{Auxiliary results}\label{auxiliary}

The following proposition is an immediate consequence of \cite[Lemma~2.2]{BAMBERG2022107}; see also \cite[Theorem~3.4$^\prime$]{Wielandt1964}. For completeness, we include its proof.

\begin{lemma}\label{l1}
    Let $G$ be the Galois group (over $\Q$) of an irreducible polynomial $f(t) \in \Q[t]$ and let $H<G$ be the Sylow $p$--subgroup. If the degree of $f(t)$ is $d=p^m$, where the integer exponent of a prime $p$ is $m \geq 1$,  then $H$ acts transitively on the roots $\al_1$, $\dots$, $\al_d$ of $f(t)$. If the degree $d=2p^m$, $p \geq 3$, $m \geq 1$, then there are exactly two orbits of size $p^m$ of roots $\al_j$ under the action of $H$. 
\end{lemma}

\begin{proof}[Proof of Lemma \ref{l1}]
We first treat the case $d=p^m$. The polynomial $f(t)$ is irreducible over $\Q$, therefore, the group $G$ acts on its roots transitively. Take any root $\al_j$ of $f(t)$ and consider its stabilizer subgroup $K=\text{Stab}(\al_j)$ in $G$. By the orbit stabilizer theorem, $p^m = [G : K]$. Let the Sylow $p$-subgroup of $G$ be of the order $|H|=p^n$, where $n \geq m$ is the exact exponent of $p$ in the prime factorization of $|G|$. From the subgroup index relation $|G|=[G:K]\cdot|K|$, the exact power of $p$ that divides $|K|$ is $p^{n-m}$. On the other hand, the size of the $H$--orbit of $\al_j$ is $[H: H \cap K] = |H|/|H \cap K|=p^n/|H \cap K|$. The exact power of $p$ dividing $|H \cap K|$ is at most $p^{n-m}$. Therefore, $[H : H \cap K]$ is divisible by $p^n/p^{n-m}=p^m$. This means every one of the $d=p^m$ roots $\al_j$ belongs to the $H$--orbit of $\al$, so $H$ acts transitively on the roots of $f(t)$.

\noindent Consider the case $d=2p^m$ with $p \geq 5$. By the same argument as in the previous case, we see that the size of the $H$--orbit of any root $\al_j$ is a multiple of $p^m$. The orbit size also must be a power of $p$, since the index $[H:H \cap K]$ divides the order of $H$, $|H|=p^n$. As there are $2p^m$ different roots $\al_j$ in total, then orbit size must be equal to $p^m$, and there must be exactly two such orbits.
\end{proof}

\begin{lemma}[Chapter I Lemma~6.3 in \cite{Lang2002}]\label{lefpte}
Let $p$ be a prime number and let $G$ be a $p$-group acting on a finite set $X$. If $|X|$ is divisible by $p$, then the number of fixed points of $G$ is also divisible by $p$. 
\end{lemma}

\section{Preliminary setup before proofs}\label{setup}

Let $\al_1$, $\al_2$, $\dots$, $\al_d$ be the roots of the irreducible polynomial $f(t) \in \Q[t]$ of degree $d>1$. 
 We define the set of integer linear relations among the roots of $f(t)$ as:
\[
A = \left\{ (a_1, \dots, a_d) \in \mathbb{Z}^d: \sum_{j=1}^d a_j \al_j = 0\right\}.
\]
Because the sum of two relations is a relation, and an integer multiple of a relation is a relation, $A$ is a $\mathbb{Z}$-submodule of $\mathbb{Z}^d$. We may assume that $A \ne \{O\}$, for otherwise, only the trivial relation with $a_1=\dots=a_d=0$ holds, and the statements of our theorems in this case are trivial. Note that the module $A$ is saturated: for every vector $v \in \Z^d$, the relation $kv \in A$ for a non-zero integer $k$ means $v \in A$. Also, $A$ must be of $\text{rank}(A)<d$: for otherwise, one immediately obtains $\al_1 = \dots = \al_d = 0$, which contradicts the irreducibility of $f(t)$ of degree $d>1$.

 Let $G$ be the Galois group (over $\Q$) of the polynomial $f(t)$. The automorphism $g \in G$ maps each root $\al_j$ to $\al_{g(j)}$ and sends the linear combination $\sum_{j=1}^d a_j\al_j$ to $\sum_{j=1}^d a_{g^{-1}(j)}\al_j$. Let $G$ act on vectors $v \in \Z^d$ by
\[
g \cdot (v_1, v_2, \dots, v_d) := \left(v_{g^{-1}(1)}, v_{g^{-1}(2)}, \dots, v_{g^{-1}(d)}\right) \in \Z^d.
\]
This defines the (left) action of $G$ on $\Z^d$ that preserves $A$. Next, we reduce it modulo a prime number $p$:
\[
\Z^d \mapsto \Z^d / p\Z^d \cong \F_p^d, \qquad \overline{v} := v + p\Z^d.
\]
Here, $\F_p$ denotes the finite field with $p$ elements, and $\F_p^d$ is a $d$--dimensional vector space over $\F_p$. The action of $G$ on $\Z^d$ descends to $\Z^d / p\Z^d$ by setting $g \cdot \overline{v} := \overline{g \cdot v}$. Since $A$ is saturated, by the rank-nullity theorem, its image $\bar{A}$ under reduction $\pmod{p}$ is a proper subspace of $\F_p^d$ of dimension $0 < \dim_{\F_p}\bar{A} \leq\text{ rank}(A)<d$. Now let
\[
\bar{A}^{\perp} := \left\{\overline{u} \in \F_p^d: \overline{u} \cdot \overline{v} = \overline{0} \in \F_p \text{ for every } \overline{v} \in \bar{A}\right\}.
\]
be the orthogonal complement of $\bar{A}$ in $\F_p^d$ with respect to the standard inner product:
\[
\langle \overline{u}, \overline{v} \rangle=\overline{(u_1, u_2, \dots, u_d)} \cdot \overline{(v_1, v_2, \dots, v_d)} = \overline{u_1 v_1 + \dots +u_dv_d}.
\]
The standard dot product is invariant with respect to the coordinate permutations, because $\sum_{i=1}^d u_{g^{-1}(i)} v_{g^{-1}(i)} = \sum_{j=1}^d u_j v_j$. It follows that $\bar{A}$ and $\bar{A}^{\perp}$ are $G$--invariant, proper subspaces of $\F_p^d$. By restricting to a subgroup, they are also $H$--invariant for any subgroup $H < G$.

\section{Proofs of main results}\label{proofs}

\begin{proof}[Proof of Theorem \ref{t1}]
Let $G$ be the Galois group of $f(t)$ over $\Q$ and let $H < G$ be its Sylow $p$--subgroup. Consider the module $\bar{A}^{\perp}$ defined in Section \ref{setup}. The set $\bar{A}^\perp$ is a proper $H$--invariant subspace of $\F_p^d$ of dimension $r$, $0 < r < d$. Therefore, the number of elements $\abs{\bar{A}^\perp}=p^r>1$. By applying Lemma~\ref{lefpte}, one finds that the cardinality of the set of fixed points under this action,
\[
\bar{A}_H^\perp =\{\overline{u} \in \bar{A}^\perp: h \cdot \overline{u} = \overline{u} \text{ whenever } h \in H\},
\]
satisfies the congruence $\abs{\bar{A}_H^\perp} \equiv 0 \pmod{p}$. Since $\bar{A}_H^\perp$ contains the zero vector $\overline{(0, \dots, 0)}$, its cardinality $\abs{\bar{A}_H^\perp}$ is non-zero and must be divisible by $p$. It follows that $\bar{A}_H^\perp$ contains at least one vector $\overline{u} \ne \overline{(0, 0, \dots,0)}$ that is fixed by every automorphism $h \in H$. According to Lemma \ref{l1}, the Sylow $p$--subgroup $H$ of $G$ acts on the coordinates of the vectors $\overline{v} \in \bar{A}$ and $\bar{A}^{\perp}$ transitively. However, a non-zero vector $\overline{u} \in \bar{A}^\perp$ that is a invariant under transitive coordinate permutations must take the form $\overline{u}=\overline{(c,\dots,c)}=\overline{c} \cdot \overline{(1, \dots, 1 )}$, $\overline{c} \in \F_p$, $\overline{c} \ne \overline{0}$. Then $\overline{w}=\overline{(1, \dots, 1 )} \in \bar{A}^\perp$. Hence, for every vector $a=(a_1, \dots, a_d) \in A$,
\[
\langle \overline{w}, \overline{a} \rangle = \overline{a_1 + \dots + a_d}=\overline{0} \in \F_p.
\]
In other words, $p$ divides the sum $\sum_{j=1}a_j$.
\end{proof}

\begin{proof}[Proof of Theorem \ref{t2}]

Recall that $G$ is the Galois group acting on the roots $\al_1$, $\al_2$, $\dots$, $\al_d$  of the polynomial $f(t)$ with  $H<G$ being its Sylow $p$--subgroup.

Using the previously described construction (see Section~\ref{setup}), one defines the saturated module of relations $A \subset \mathbb{Z}^d$, reduces it modulo $p$ and takes the orthogonal complement $\bar{A}^\perp \subset \mathbb{F}_p^d$ that are invariant under the coordinate permutations by $G$ (and $H$). By the same argument as in Theorem $\ref{t1}$, there exists a vector $u \in \Z^d$, such that $\overline{u} \in \bar{A}^\perp$, $\overline{u} \ne \overline{O}_d$, is a fixed point under the action of $H$.

 According to Lemma \ref{l1}, the roots $\al_j$ of $f(t)$ fall into a two disjoint orbits under the action of $H$. Denote their root index sets by $X$ and $Y$, respectively. The coordinates of a fixed vector $\overline{u}$ then must be constant on the orbits of $H$. By reducing $u$ modulo $p$ if necessary, we may assume that there exist $x, y \in \Z$ such that $u_i = x$ for all $i \in X$ and $u_i = y$ for all $i \in Y$. Because $\overline{u}\ne \overline{O}_d$, $x$ and $y$ are not both zero modulo $p$.

For any relation $a = (a_1, \dots, a_d) \in A$, the dot product of $\langle \overline{u}, \overline{v}\rangle = \overline{0}$ in $\F_p$. Define the partial sums $S_X(a) := \sum_{j \in X} a_j$ and $S_Y(a) := \sum_{j \in Y} a_j$. The orthogonality condition yields:
\begin{equation}\label{ort}
x S_X(a) + y S_Y(a) \equiv 0 \pmod{p}.
\end{equation}
Let also $T(a) := \sum_{i=1}^d a_j = S_X(a) + S_Y(a)$. Substitute $S_Y(a) = T(a) - S_X(a)$ into the congruence:
\[
y T(a) + (x - y) S_X(a) \equiv 0 \pmod{p}.
\]
Because $A$ is $G$--invariant, for any $g \in G$, $a \in A$, $g^{-1} \cdot a = \left(a_{g(1)}, \dots, a_{g(d)}\right) \in A$. The total sum is invariant:
$$ T(g^{-1} \cdot a) = \sum_{i=1}^n a_{g(i)} = \sum_{k=1}^n a_k = T(a). $$
The partial sum over $X$ is:
$$ S_X(g^{-1} \cdot a) = \sum_{i \in X} a_{g(i)}. $$
Let $j = g(i)$. As $i$ ranges over $X$, $j$ ranges over $g(X)$. Therefore,
$$ S_X(g^{-1} \cdot a) = \sum_{j \in g(X)} a_j. $$
The substitution $a \mapsto g^{-1} \cdot a$ yields the congruence
\begin{equation}\label{congr}
y T(a) + (x - y) \sum_{j \in g(X)} a_j \equiv 0 \pmod p \quad \text{for all } g \in G.
\end{equation}
One looks at two cases regarding the dependence of $\sum_{j \in g(X)} a_j \pmod p$ on $g$.

\noindent\emph{Case 1: $\sum_{j \in g(X)} a_j \pmod p$ is not independent of $g$.} Suppose there exist $g_1, g_2 \in G$ such that $\sum_{j \in g_1(X)} a_j \not\equiv \sum_{j \in g_2(X)} a_j \pmod p$. Evaluating the relation \eqref{congr} at $g_1$ and $g_2$ and subtracting the two resulting congruences yields:
\[
(x - y) \left( \sum_{j \in g_1(X)} a_j - \sum_{j \in g_2(X)} a_j \right) \equiv 0 \pmod p.
\]
Because the second factor is non-zero modulo $p$, we must have $x - y \equiv 0 \pmod p$, which means $x \equiv y \pmod p$. Since $x$ and $y$ are not both zero modulo $p$, we have $y \not\equiv 0 \pmod p$. The congruence simplifies to $y T(a) \equiv 0 \pmod p$. This implies $T(a) \equiv 0 \pmod p$.

\noindent\emph{Case 2: $\sum_{j \in g(X)} a_j \pmod p$ is independent of $g$.} Set $C := \sum_{j \in g(X)} a_j \pmod p$ for all $g \in G$. 
Let $\Omega = \{g(X) \mid g \in G\}$ be the set of distinct translates of $X$. Let $K_X = \mathrm{Stab}_G(X) = \{ g \in G \mid g(X) = X \}$. The cardinality of $\Omega$ is the index $l = [G : K_X]$. Because $X$ is an orbit of the Sylow $p$-subgroup $H$, the group $H$ stabilizes $X$ set-wise, meaning $H \le K_X \le G$. Consequently, the index $[G : K_X]$ must divide $[G : H]$. Since $H$ is a Sylow $p$-subgroup, $[G : H]$ is coprime to $p$, which implies that $l$ is coprime to $p$.

We sum the coefficients $a_j$ over the elements of all different sets $B = g(X) \in \Omega$:
\[
S := \sum_{B \in \Omega} \sum_{j \in B} a_j.
\]
Because $G$ acts transitively on the indices $\{1, \dots, d\}$, every index $j$ is contained in exactly the same number of sets in $\Omega$. Let this multiplicity be $k$. Counting the total number of elements in all sets in two ways yields $l|X| = kd$, which gives $l p^m = k (2 p^m)$, implying $l = 2k$. Because $l$ is coprime to $p$, $k$ must also be coprime to $p$. Furthermore, this shows $l$ is an even integer. Then:
\[
S = \sum_{j=1}^d k a_j = k \sum_{j=1}^d a_j = k T(a).
\]
On the other hand, $\sum_{j \in B} a_j \equiv C \pmod p$,  so $S \equiv l C \pmod p$. Hence,
\[
k T(a) \equiv l C \equiv 2k C \pmod{p}.
\]
As $k$ is coprime to $p$, one must have $T(a) \equiv 2C \pmod p$. Putting this expression for $T(a)$ back into \eqref{congr}:
\[
2Cy + (x - y)C \equiv C(x+y) \equiv 0 \pmod{p}.
\]
This forces either $T(a) \equiv 0 \pmod{p}$ or $x + y \equiv 0 \pmod p$.

By the hypothesis of Theorem \ref{t2}, $\sum_{i=1}^d |a_i| < p$. This means that the integer $T(a)$ is strictly less than $p$ in absolute value. If $T(a) \equiv 0 \pmod{p}$, then, in view of $|T(a)| < p$, we obtain that $T(a)=0$. So $T(a)$ is even integer in this case. If $T(a) \not\equiv 0 \pmod{p}$, then $x = -y  \pmod{p}$. By the orthogonality equation \eqref{ort},
\[
x S_X(a) + y S_Y(a) \equiv x \left(S_X(a) - S_Y(a)\right) \equiv 0 \pmod{p}.
\]
Since $x$ and $y$ are not both zero modulo $p$ and $y \equiv -x \pmod p$, one has $x \not\equiv 0 \pmod p$. Thus, $S_X(a) - S_Y(a) \equiv 0 \pmod p$. The absolute value of this difference is bounded by:
\[
|S_X(a) - S_Y(a)| \le |S_X(a)| + |S_Y(a)| \le \sum_{i \in X} |a_i| + \sum_{i \in Y} |a_i| = \sum_{i=1}^d |a_i| < p.
\]
Thus, the difference must be zero identically, which means $S_X(a) = S_Y(a)$. Then the total sum $T(a) = S_X(a) + S_Y(a) = 2 S_X(a)$ is, again, even. This concludes the proof of Theorem 2.
\end{proof}

\begin{proof}[Proof of Theorem \ref{p2}]
It is an immediate consequence of Theorem~\ref{t1} and Theorem~\ref{t2}.
\end{proof}

\begin{proof}[Proof of Theorem \ref{p1}]
Every  positive integer which is strictly smaller than 20 and which is not a multiple of $3$, namely,
\[
2,\; 4=2^2, \; 5, \; 7, \; 8=2^3, \; 10=2\cdot5, \; 11, \; 13, \; 14=2\cdot 7,\; 16=2^4, \; 17, \; 19
\]
is of the form $p^m$, $p \ne 3$, or of the form $2p^m$, $p \geq 5$. 
Applying Theorem~\ref{p2} we obtain that none of these integers equals $d$. Hence, $d\geq 20$. 
In view of the example of Stong \eqref{example20}, we obtain that $d=20$.
\end{proof}

\section{Open cases}\label{open}

If tree distinct algebraic conjugates $\al_1$, $\al_2$, $\al_3$ of degree $d>1$ (over $\Q$) satisfy the linear relation $\al_1+\al_2+\al_3=0$, then, for each $s \in \N$, one can produce new algebraically conjugate numbers $\be_1$, $\be_2$, $\be_3$ of degree $ds$ (all degrees are over $\Q$) that also satisfy the same linear relation $\be_1 + \be_2 + \be_3=0$. Indeed, it suffices to consider the product $\be_j=\al_j(a+\ga)$, $j=1,2,3$, where an algebraic number $\ga$ is of degree $s$ over $\Q$ and over $\Q(\al_1)$ and $a$ is a rational integer. Then for sufficiently large $a$ the degree of $\be_j$ is $ds$. In this way, starting with a cubic polynomial $f(t)=t^3-t-1$, one can construct new algebraically conjugate solutions $\al_1$, $\al_2$, $\al_3$ to the equation $\al_1+\al_2+\al_3=0$ for each degree $d=3s$, $s \in \N$. Likewise, starting with an example of Stong \eqref{example20} of degree $20$, one can produce new solutions of degree $d=20s$, for each $s \in \N$. If we take a list of all integers in the interval $[2, 100]$ and cross out all multiples of $3$, $20$ and all the degrees forbidden by Theorem \ref{p2} (prime powers and two times the power of an odd prime) we are left with the following values of degree $d$:
\[
28, 35, 44, 52, 55, 56, 65, 68, 70, 76, 77, 85, 88, 91, 92, 95.
\]
for which neither solutions to the equation $\al_1+\al_2+\al_3=0$ in conjugate algebraic numbers $\al_1$, $\al_2$, $\al_3$ of degree $d$ are currently known, nor the non-existence of such solutions have been established.

\begin{acknowledgment}
    We thank the anonymous referee for suggesting Theorems~\ref{t1}
and~\ref{t2} and for providing their proofs. The referee informed us
that these proofs were generated with the assistance of
\texttt{ChatGPT} (OpenAI, GPT-5.5 Pro). We have independently verified
and revised the proofs and take full responsibility for their
correctness. Our original proof of Theorem~\ref{p1} used a different
approach and established only the special case \(d=16\).
\end{acknowledgment}

\printbibliography

\end{document}